\documentclass[12pt,a4paper]{amsart}
\usepackage[abbrev]{amsrefs}
\usepackage{amssymb,amsmath}

\theoremstyle{plain}
  \newtheorem{thm}{Theorem}[section]
  \newtheorem{lem}[thm]{Lemma}
  
  \newtheorem{cor}[thm]{Corollary}
  \newtheorem{prop}[thm]{Proposition}

\theoremstyle{definition}
  \newtheorem{defn}[thm]{Definition}

  \newtheorem{prob}[thm]{Problem}

\theoremstyle{remark}

  \newtheorem*{ack}{Acknowledgment}

\numberwithin{equation}{section}

\DeclareMathOperator{\diam}{diam}
\DeclareMathOperator{\supp}{supp}

\DeclareMathOperator{\ObsDiam}{ObsDiam}

\DeclareMathOperator{\id}{id}
\DeclareMathOperator{\dconc}{{\it d}_{{\rm conc}}}

\DeclareMathOperator{\dKF}{{\it d}_{{\rm KF}}}
\DeclareMathOperator{\dP}{{\it d}_{{\rm P}}}

\DeclareMathOperator{\dH}{{\it d}_{{\rm H}}}

\DeclareMathOperator{\Aut}{Aut}

\DeclareMathOperator{\Dom}{Dom}
\DeclareMathOperator{\Imag}{Im}

\newcommand{\Lip}{\mathcal{L}{\it ip}}

\newcommand{\cX}{\mathcal{X}}

\newcommand{\eq}{{\rm eq}}

\newcommand{\field}[1]{\mathbb{#1}}

\newcommand{\C}{\field{C}}
\newcommand{\R}{\field{R}}

\newcommand{\Z}{\field{Z}}

\begin{document}

\title[Convergence of group actions]
{Convergence of group actions\\ in metric measure geometry}

\begin{abstract}
We generalize the box and observable distances
to those between metric measure spaces with group actions,
and prove some fundamental properties.
As an application, we obtain an example of a sequence of lens spaces
with unbounded dimension converging to the cone of
the infinite-dimensional complex projective space.
Our idea is to use the theory of mass-transport.
\end{abstract}

\thanks{This work was supported by JSPS KAKENHI Grant Number 19K03459 and 19J10866.}

\author{Hiroki Nakajima}
\address{Institute for Excellence in Higher Education, Tohoku University,
Sendai 980-8576, Japan}
\email{hiroki.nakajima.a1@tohoku.ac.jp}

\author{Takashi Shioya}
\address{Mathematical Institute, Tohoku University, Sendai 980-8578, JAPAN}
\email{shioya@math.tohoku.ac.jp}

\date{\today}

\keywords{box distance, observable distance, group action}

\subjclass[2010]{53C23}

\maketitle

\section{Introduction}

Fukaya \cite{Fk:Theory} first introduced the concept of
convergence of metric spaces with isometric group actions,
which was applied to the study of the collapsing theory
and fundamental group (see \cite{FY:f-gp}).
Recently, it is further developed in
the study of dynamical system
(see for example \cite{AR:top-stab}).
In this paper, we study convergence of metric measure spaces
(mm-spaces)
with measure preserving and isometric group actions.
We generalize the box and observable distances
due to Gromov \cite{Gmv:green}
to those between mm-spaces with group actions,
and study their properties.
Our main results are the following.

\begin{enumerate}
\item We define the equivariant box and observable metric and
prove that they are really metrics
(Theorem \ref{thm:box-dist} and \ref{thm:dconc}).
\item We prove that if a sequence of metric measure spaces with group actions
converges with respect to the equivariant box/observable metric, 
then so does the sequence of the quotient spaces
(Theorem \ref{thm:box-quot} and \ref{thm:conc-quot}).
\item We prove that the projection from the set of mm-spaces
with group actions to the set of mm-spaces
(i.e., the projection dropping group actions)
is a proper map (Theorem \ref{thm:subconv}).
\item As an application, we have an example of a sequence of lens spaces with unbounded dimension converging to the cone of
the infinite-dimensional complex projective space in the box metric
(Theorem \ref{thm:lens}).
\end{enumerate}

For (1), we remark that the nondegeneracy of the metrics
is not so easy to prove even for mm-spaces without
group actions.
Our idea is to use the mass-transport theory to formulate
the equivariant box and observable metrics.
Such a formulation enables us to prove the nondegeneracy.

The key point of (2) is to prove the precompactness
of groups acting on metric measure spaces with respect to the Ky Fan metric
by using the thick-thin decomposition of the spaces (see Lemmas \ref{lem:Xneps} and \ref{lem:Gn-precpt}).
It is delicate that the groups do not Gromov-Hausdorff converge
to the group acting on the limit space with respect to the
Ky Fan metric.

(1)--(3) was proved in the case of metric spaces
without measures in \cites{Fk:Theory,FY:f-gp}.
We emphasize that the proofs in our case are much more delicate.

(4) is an interesting example of convergence
which is only obtained for metric measure spaces
and not for metric spaces.

\begin{ack}
The authors would like to thank Daisuke Kazukawa
for pointing out some minor mistakes in an earlier version
of this paper and for valuable comments.
\end{ack}

\section{Preliminaries}

In this paper, a \emph{metric measure space},
an \emph{mm-space} for short, means a complete separable
metric space with a Borel probability measure.
We assume that the measure of an mm-space has
full-support unless otherwise stated.
For an mm-space $X$, the metric and measure of $X$ are
respectively denoted by $d_X$ and $\mu_X$.
Sometimes, the triple $(X,d_X,\mu_X)$ is called an mm-space.
Two mm-spaces $X$ and $Y$ are \emph{mm-isomorphic}
to each other if there exists an isometry $f : X \to Y$
with $f_*\mu_X = \mu_Y$, where $f_*\mu_X$ is the
push-forward measure of $\mu_X$ by $f$.

For finitely many sets
$X_i$, $i=1,2,\dots,n$, and for numbers $i_1,i_2,\dots,i_k \in \{1,2,\dots,n\}$, let
$p_{i_1i_2\dots i_k} : \prod_{i=1}^n X_i \to \prod_{j=1}^k X_{i_j}$
be the projection defined by
\[
p_{i_1i_2\dots i_k}(x_1,x_2,\dots,x_n) := (x_{i_1},x_{i_2},\dots,x_{i_k}), \qquad x_i \in X_i,\ i=1,2,\dots,n.
\]

Let $X$, $Y$, and $Z$ be mm-spaces.
Denote by $\Pi(X,Y)$ the set of couplings (or transport plan)
between $\mu_X$ and $\mu_Y$.
It is well-known that $\Pi(X,Y)$ is compact with respect to
the Prokhorov metric $\dP$.
For $\sigma \in \Pi(X,Y)$ and $\tau \in \Pi(Y,Z)$,
the \emph{gluing} $\tau \bullet \sigma \in P(X \times Y \times Z)$
is defined by
\[
\tau \bullet \sigma(A \times B \times C) :=
\int_B \sigma_y(A) \tau_y(C) \, d\mu_Y(y),
\]
where $\{\sigma_y\}_{y \in Y}$ and $\{\tau_y\}_{y \in Y}$
are the disintegrations respectively of
$\sigma$ and $\tau$ for the projections
$p_2 : X \times Y \to Y$ and $p_1 : Y \times Z \to Y$.
It holds that $(p_{12})_*(\tau\bullet\sigma) = \sigma$ and
$(p_{23})_*(\tau\bullet\sigma) = \tau$.
We also define
$\tau\circ\sigma := (p_{13})_*(\tau\bullet\sigma)
\in \Pi(X,Z)$.
We observe the associativity $(\upsilon \circ \tau) \circ \sigma = \upsilon \circ (\tau \circ \sigma)$.
For subsets $S \subset X \times Y$ and $T \subset Y \times Z$,
we define
$T \bullet S := (X \times T) \cap (S \times Z)$,
$T \circ S := p_{13}(T \bullet S)$,
$\Dom(S) := p_1(S)$,
$S^{-1} := \{\,(y,x) \mid (x,y) \in S\,\}$,
and $\Imag(S) := \Dom(S^{-1})$.
If $S$ and $T$ are Borel, then
$\Dom(S)$, $T\circ S$, and $\Dom(T\circ S)$ are all analytic sets
and measurable with respect to the completion of
any Borel measure.
We denote the completion of a Borel measure
by the same notation.
Let $\sigma \in \Pi(X\times Y)$ and $\tau \in \Pi(Y \times Z)$.
Since
$\mu_X(\Dom(S)) = \sigma(\Dom(S) \times Y) \ge \sigma(S)$,
we see
\begin{align*}
&\mu_X(\Dom(T \circ S)) \ge (\tau\circ\sigma)(T \circ S)
\ge (\tau\bullet\sigma)(T \bullet S)\\
&= 1-(\tau\bullet\sigma)((X \times T)^c \cup (S \times Z)^c)
\ge \tau(T) + \sigma(S) - 1,
\end{align*}
which is often used in this paper.

\section{Equivariant box distance}

In this section, we introduce the box distance
between mm-spaces with mm-actions
and prove a triangle inequality.

Let $X$ be an mm-space and $G$ a group.
An \emph{mm-action} of $G$ on $X$ is, by definition,
isometric action of $G$ preserving $\mu_X$.
Denote by $\Aut(X)$ the group of isometries of $X$
preserving $\mu_X$.
If an mm-action of $G$ on $X$ is effective,
then $G$ is identified with a subgroup of $\Aut(X)$.

\begin{defn}
Let $X$ and $Y$ be two mm-spaces and
let $G \subset \Aut(X)$ and $H \subset \Aut(Y)$
be subgroups.
We say that $(X,G)$ is \emph{equivariantly mm-isomorphic}
 to $(Y,H)$
and write $(X,G) \cong (Y,H)$ if there exists
an mm-isomorphism from $X$ to $Y$ that is
equivariant for some group-isomorphism from $G$ to $H$.
Denote by $\cX_\eq$ the equivariant mm-isomorphism classes
of $(X,G)$, where $X$ is any mm-space and $G$
is any \emph{closed} subgroup of $\Aut(X)$.
\end{defn}

We sometimes write $(X,G) \in \cX_\eq$,
which means that $(X,G)$ is a representative
of a class in $\cX_\eq$.

The following lemma is obvious and
the proof is omitted.

\begin{lem} \label{lem:Lip-conv}
Let $X$ be a complete separable metric space and $Y$ a metric space.
Let $\varphi_i : X \to Y$, $i=1,2,\dots$, be $C$-Lipschitz maps
for a constant $C \ge 0$,
and $X_0 \subset X$ a dense subset.
Then, if $\varphi_i$ converges pointwise on $X_0$ as $i\to\infty$, then $\varphi_i$ converges pointwise on the whole $X$.
\end{lem}

We denote the Prokhorov metric by $\dP$
and the Ky Fan metric by $\dKF$.

\begin{lem} \label{lem:Aut-conv}
Let $X$ be an mm-space.
For a sequence of elements of $\Aut(X)$,
the following {\rm(i)---(iii)} are equivalent to each other.
\begin{enumerate}
\item[(i)] The sequence converges uniformly on compact sets.
\item[(ii)] The sequence converges pointwise.
\item[(iii)] The sequence converges in measure.
\end{enumerate}
Moreover,  $\Aut(X)$ is compact with respect to
the Ky Fan metric.
\end{lem}

\begin{proof}
Let us first prove the relative compactness of the orbit
$\Aut(X)x = \{\varphi(x)\mid\varphi\in\Aut(X)\}$
for each $x \in X$.
In fact, for any $\varepsilon > 0$ and $\varphi \in \Aut(X)$,
since $\mu_X(B_{\varepsilon/2}(x)) = \mu_X(B_{\varepsilon/2}(\varphi(x)))$, any $\varepsilon$-discrete net in $\Aut(X)x$
 is finite and so $\Aut(X)x$ is precompact.
Since $X$ is complete, the orbit $\Aut(X)x$ is
relatively compact.

The implication (i) $\Rightarrow$ (ii) is clear.

We prove (ii) $\Rightarrow$ (i).
Let a sequence $\{\varphi_i\} \subset \Aut(X)$ converge
pointwise to a map $\varphi$ and let $K \subset X$ be compact.
If $\{\varphi_i\}$ does not converge uniformly to $\varphi$
on $K$, then there are a subsequence
$\{\varphi_{i_j}\}$ and points $x_j,x \in K$ such that
\begin{equation} \label{eq:xj}
\lim_{j\to\infty} x_j = x\quad\text{and}\quad
\inf_j d_X(\varphi_{i_j}(x_j),\varphi(x)) > 0.
\end{equation}
Since $\Aut(X)$ is equicontinuous and any orbit is relatively compact, Ascoli-Arzel\`a theorem implies that
$\{\varphi_{i_j}\}$ has a subsequence converging uniformly
on $K$.  However, the limit must be $\varphi$,
which contradicts \eqref{eq:xj}.
We therefore obtain (i).

The implication (i) $\Rightarrow$ (iii) is clear.

To prove (iii) $\Rightarrow$ (ii),
we assume that a sequence $\{\varphi_i\} \subset \Aut(X)$
converges in measure to $\varphi$.
For any $\varepsilon > 0$, setting
$\Phi_{i,\varepsilon} :=
\{\,\xi \in X \mid d_X(\varphi_i(\xi),\varphi(\xi)) < \varepsilon\,\}$,
we have $\mu_X(\Phi_{i,\varepsilon}) \to 1$ as $i\to\infty$.
For any $x \in X$ and $\varepsilon > 0$,
the ball $B_\varepsilon(x)$ intersects $\Phi_{i,\varepsilon}$
for all sufficiently large $i$.
There is a point $x' \in B_\varepsilon(x)$ with
$d_X(\varphi_i(x'),\varphi(x')) < \varepsilon$.
Since $d_X(\varphi(x),\varphi(x')) = d_X(x,x') < \varepsilon$
and $d_X(\varphi_i(x),\varphi_i(x')) = d_X(x,x') < \varepsilon$,
a triangle inequality implies
$d_X(\varphi_i(x),\varphi(x)) < 3\varepsilon$.
Thus, $\varphi_i(x)$ converges to $\varphi(x)$.
(ii) has been obtained.

We have proved that
(i)--(iii) are equivalent to each other.
The rest is to prove the compactness of $\Aut(X)$.

Take any sequence $\{\varphi_i\} \subset \Aut(X)$
and a dense countable subset $X_0 \subset X$.
Since any $\Aut(X)$-orbit is relatively compact,
$\{\varphi_i(x_0)\}$ for any fixed $x_0 \in X_0$
has a convergent subsequence.
By a standard diagonal argument,
there is a subsequence $\{\varphi_{i_j}\}$
that converges pointwise on $X_0$.
By Lemma \ref{lem:Lip-conv},
$\varphi_{i_j}$ converges pointwise (and so, in measure) to
a map $\varphi : X \to X$ on $X$.
The map $\varphi$ is an isometry.
It follows from \cite{Sy:mmg}*{Lemma 1.26} that
$\dP(\mu_X,\varphi_*\mu_X) =
\dP((\varphi_{i_j})_*\mu_X,\varphi_*\mu_X)
\le \dKF(\varphi_{i_j},\varphi) \to 0$ as $j\to\infty$,
which implies
$\varphi_*\mu_X = \mu_X$ and so $\varphi \in \Aut(X)$.
This completes the proof of the lemma.
\end{proof}

\begin{defn}[Equivariant box metric]
Let $(X,G), (Y,H) \in \cX_\eq$.
For a nonempty subset $S \subset X \times Y$ and
$\gamma_1,\gamma_2 \in X^X \sqcup Y^Y$, we define
\[
d^S(\gamma_1,\gamma_2)
:=\sup_{S\times S} |\,f_{\gamma_1} - f_{\gamma_2}\,|
\quad (\le +\infty),
\]
where, for $(x_i,y_i) \in S$, $i=1,2$,
\[
f_\gamma((x_1,y_1),(x_2,y_2)) := 
\begin{cases}
d_X(\gamma(x_1),x_2) & \text{for $\gamma \in X^X$,}\\
d_Y(\gamma(y_1),y_2) & \text{for $\gamma \in Y^Y$.}
\end{cases}
\]
We set $d^\emptyset(\gamma_1,\gamma_2) := 0$.
For $\pi \in \Pi(X,Y)$ and
$\gamma_1,\gamma_2 \in X^X \sqcup Y^Y$, we define
\[
d^\pi(\gamma_1,\gamma_2) :=
\inf_S \max\{1-\pi(S), d^S(\gamma_1,\gamma_2)\},
\]
where $S$ runs over all closed subsets of the support
`$\supp\pi$' of $\pi$.
Both $d^S$ and $d^\pi$ are extended pseudo-metrics,
where `extended' means that the metrics take values
in $[\,0,+\infty\,]$.
Denoting by $\square^\pi((X,G),(Y,H))$
the Hausdorff distance between $G$ and $H$ with respect to
$d^\pi$,
we define the \emph{equivariant box distance} between
$(X,G)$ and $(Y,H)$ to be
\[
\square((X,G),(Y,H)) := \inf_{\pi \in \Pi(X,Y)} \square^\pi((X,G),(Y,H)),
\]
and also define
\[
\square(X,Y) := \square((X,\{\id_X\}),(Y,\{\id_Y\})).
\]
\end{defn}

Note that $\square \le 1$ because of $d^\emptyset = 0$.

It follows from \cite{Nk:thesis}*{Section 6}
(and also \cite{Nk:box-obs}) that
$\square(X,Y)$ coincides with
the original box distance between $X$ and $Y$
defined by Gromov \cite{Gmv:green}.

In this and the next sections,
we prove that $\square$ is a metric on $\cX_\eq$.
The following is obvious by the definition of $\square^\pi$.

\begin{lem} \label{lem:box}
Let $(X,G), (Y,H) \in \cX_\eq$ and $\varepsilon > 0$.
Then, the inequality\\
$\square^\pi((X,G),(Y,H)) < \varepsilon$
is equivalent to both {\rm(i)} and {\rm(ii)} being satisfied.
\begin{enumerate}
\item[(i)] For any element $g \in G$
there exist an element $h \in H$
and a closed subset $S \subset \supp\pi$ with $\pi(S) > 1-\varepsilon$ such that
\begin{equation} \label{eq:box}
d^S(g,h) < \varepsilon.
\end{equation}
\item[(ii)] For any element $h \in H$
there exist an element $g \in G$
and a closed subset $S \subset \supp\pi$ with $\pi(S) > 1-\varepsilon$ such that \eqref{eq:box} holds.
\end{enumerate}
\end{lem}

\begin{lem} \label{lem:box-triangle}
$\square$ satisfies a triangle inequality on $\cX_\eq$.
\end{lem}

\begin{proof}
Let us prove $\square((X,G),(Z,K)) \le \square((X,G),(Y,H)) + \square((Y,H),(Z,K))$ for $(X,G),(Y,H),(Z,K) \in \cX_\eq$.
Take any $\varepsilon_1$, $\varepsilon_2$ with
$\varepsilon_1 > \square((X,G),(Y,H))$ and
$\varepsilon_2 > \square((Y,H),(Z,K))$.
There are $\pi_1 \in \Pi(X,Y)$ and
$\pi_2 \in \Pi(Y,Z)$
such that
$\square^{\pi_1}((X,G),(Y,H))<\varepsilon_1$ and
$\square^{\pi_2}((Y,H),(Z,K))<\varepsilon_2$.
We set $\pi := \pi_2 \circ \pi_1$.
It suffices to prove
$\square^\pi((X,G),(Z,K)) \le \varepsilon_1 + \varepsilon_2$.
Take any $g \in G$.
Since $\square^{\pi_1}((X,G),(Y,H))<\varepsilon_1$
and by Lemma \ref{lem:box},
there are a closed subset
$S_1 \subset \supp\pi_1$ with $\pi_1(S_1) > 1-\varepsilon_1$
and $h \in H$ such that
$d^{S_1}(g,h) < \varepsilon_1$.
In the same way,
$\square^{\pi_2}((Y,H),(Z,K))<\varepsilon_2$ implies
that there are a closed subset $S_2 \subset \supp\pi_2$
with $\pi_2(S_2) > 1-\varepsilon_2$ and $k \in K$
such that
$d^{S_2}(h,k) < \varepsilon_2$.
Let $S := \overline{S_2 \circ S_1}$, where
the upper bar means the closure.
For any $(x_i,z_i) \in S_2 \circ S_1$, $i=1,2$,
there is $y_i \in Y$ such that $(x_i,y_i) \in S_1$
and $(y_i,z_i) \in S_2$.  We have
\begin{align*}
& |\,d_X(g(x_1),x_2) - d_Z(k(z_1),z_2)\,|\\
&\le |\,d_X(g(x_1),x_2) - d_Y(h(y_1),y_2)\,|
+ |\,d_Y(h(y_1),y_2) - d_Z(k(z_1),z_2)\,|\\
&\le \varepsilon_1 + \varepsilon_2.
\end{align*}
Since these inequalities are kept under the limit,
these hold for any $(x_i,z_i) \in S$,
i.e., we have $d^S(g,k) \le \varepsilon_1 + \varepsilon_2$.
By Lemma \ref{lem:box}, we obtain
$\square^\pi((X,G),(Z,K)) \le \varepsilon_1 + \varepsilon_2$.
In the same way,
for any $k \in K$ there is $g \in G$ such that
$d^\pi(g,k) \le \varepsilon_1 + \varepsilon_2$.
We thus obtain
$\square((X,G),(Z,K)) \le \square^\pi((X,G),(Z,K)) \le \varepsilon_1 + \varepsilon_2$.
This completes the proof.
\end{proof}

The nondegeneracy of $\square$ on $\cX_\eq$ is proved
in the next section.

\begin{prop} \label{prop:box-boxeq}
For any $(X,G), (Y,H) \in \cX_\eq$, we have
\[
\square(X,Y) \le 2\,\square((X,G),(Y,H)).
\] 
\end{prop}

\begin{proof}
Let $\varepsilon > \square((X,G),(Y,H))$.
We apply Lemma \ref{lem:box}(ii) for $h := \id_Y$.
Then, there are $g \in G$ and a closed subset
$S\subset \supp\pi$ satisfying $\pi(S) \ge 1-\varepsilon$ and \eqref{eq:box}.
In particular,
we have $d_X(g(x),x) \le \varepsilon$ for any $x \in p_1(S)$.
A triangle inequality and \eqref{eq:box} together
imply that, for any $(x_i,y_i) \in S$, $i=1,2$,
\[
|\,d_X(x_1,x_2) - d_Y(y_1,y_2)\,|
\le |\,d_X(g(x_1),x_2) - d_Y(y_1,y_2)\,| + \varepsilon \le 2\varepsilon.
\]
This completes the proof.
\end{proof}

\section{Equivariant observable distance}

In this section, we introduce the equivariant observable metric
on $\cX_\eq$ and prove that it is a really metric.
As an application, we prove the nondegeneracy
of the equivariant box metric.
 
Denote by $\Lip_1(X)$ the set of real valued
$1$-Lipschitz functions on a metric space $X$.

\begin{defn}[Equivariant observable metric] \label{defn:dconc}
Let $(X,G), (Y,H) \in \cX_\eq$, $\pi \in \Pi(X,Y)$,
$g \in G$,  and $h \in H$.
For $f \in \Lip_1(X)$ and $f' \in \Lip_1(Y)$,
we define
\[
\rho^\pi_{g,h}(f,f') :=
\max\{\dKF^\pi(f\circ p_1,f'\circ p_2),
\dKF^\pi(f\circ g\circ p_1,f'\circ h\circ p_2)\},
\]
where $\dKF^\pi$ denotes the Ky Fan metric on the set of
Borel measurable functions on $X \times Y$ with respect to $\pi$.
Let $\rho^\pi(g,h)$ denote the Hausdorff distance between
$\Lip_1(X)$ and $\Lip_1(Y)$ with respect to $\rho^\pi_{g,h}$,
and $\dconc^\pi((X,G),(Y,H))$ the Hausdorff distance between
$G$ and $H$ with respect to $\rho^\pi$.
Define the \emph{equivariant observable distance} between
$(X,G)$ and $(Y,H)$ to be
\[
\dconc((X,G),(Y,H)) := \inf_{\pi \in \Pi(X,Y)}
\dconc^\pi((X,G),(Y,H)),
\]
and also define
\[
\dconc(X,Y) := \dconc((X,\{\id_X\}),(Y,\{\id_Y\})).
\]
\end{defn}

Note that, by the same trick in the definition of
equivariant box distance,
we are able to extend $\rho^\pi$ to a pseudo-metric on
$G \sqcup H$,
however we do not need such the extension to obtain
the Hausdorff distance and to define $\dconc^\pi$.

It follows from \cite{Nk:thesis}*{Section 6} (and also \cite{Nk:box-obs})
that $\dconc(X,Y)$ coincides with
the original observable distance between $X$ and $Y$
defined by Gromov \cite{Gmv:green}.
It is obvious that
$\dconc(X,Y) \le \dconc((X,G),(Y,H))$.

\begin{lem} \label{lem:dconc-triangle}
Let $(X,G), (Y,H), (Z,K) \in \cX_\eq$,
$\pi_1 \in \Pi(X,Y)$, and $\pi_2 \in \Pi(Y,Z)$.
\begin{enumerate}
\item For any $g \in \Aut(X)$, $h \in \Aut(Y)$, and $k \in \Aut(Z)$, we have
\[
\rho^{\pi_2\circ\pi_1}(g,k) \le \rho^{\pi_1}(g,h) + \rho^{\pi_2}(h,k).
\]
\item We have
\begin{align*}
\dconc^{\pi_2\circ\pi_1}((X,G),(Z,K))
\le \dconc^{\pi_1}((X,G),(Y,H)) + \dconc^{\pi_2}((Y,H),(Z,K)).
\end{align*}
\item $\dconc$ satisfies a triangle inequality on $\cX_\eq$.
\end{enumerate}
\end{lem}

\begin{proof}
We prove (1).
Take any $\varepsilon_i$ with
$\rho^{\pi_1}(g,h) < \varepsilon_1$ and $\rho^{\pi_2}(h,k) < \varepsilon_2$.
For any $f \in \Lip_1(X)$ there is $f' \in \Lip_1(Y)$ such that
$\dKF^{\pi_2\bullet\pi_1}(f\circ p_1,f'\circ p_2)
= \dKF^{\pi_1}(f\circ p_1,f'\circ p_2) < \varepsilon_1$ and
$\dKF^{\pi_2\bullet\pi_1}(f\circ g\circ p_1,f'\circ h\circ p_2)
= \dKF^{\pi_1}(f\circ g\circ p_1,f'\circ h\circ p_2) < \varepsilon_1$.
For the $f'$ there is $f'' \in \Lip_1(Z)$ such that
$\dKF^{\pi_2\bullet\pi_1}(f'\circ p_2,f''\circ p_3) = \dKF^{\pi_2}(f'\circ p_1,f''\circ p_2) < \varepsilon_2$ and
$\dKF^{\pi_2\bullet\pi_1}(f'\circ h\circ p_2,f''\circ k\circ p_3)
= \dKF^{\pi_2}(f'\circ h\circ p_1,f''\circ k\circ p_2) < \varepsilon_2$.
A triangle inequality implies that
$\dKF^{\pi_2\circ\pi_1}(f\circ p_1,f''\circ p_2)
= \dKF^{\pi_2\bullet\pi_1}(f\circ p_1,f''\circ p_3)
< \varepsilon_1 + \varepsilon_2$ and
$\dKF^{\pi_2\circ\pi_1}(f\circ g\circ p_1,f''\circ k\circ p_2)
= \dKF^{\pi_2\bullet\pi_1}(f\circ g\circ p_1,f''\circ k\circ p_3)
< \varepsilon_1 + \varepsilon_2$.
In the same way, for any $f'' \in \Lip_1(Z)$ there is
$f \in \Lip_1(X)$ satisfying the last inequality.
Therefore we obtain
$\rho^{\pi_2\circ\pi_1}(g,k) \le \varepsilon_1 + \varepsilon_2$.
(1) has been proved.

(2) follows from (1) by a straightforward discussion.

(3) follows from (2).

This completes the proof.
\end{proof}

We need some lemmas for the proof of
the nondegeneracy of $\dconc$.

\begin{lem}[\cite{Nk:thesis}*{Proposition 6.13}] \label{lem:dKF-dP}
Let $X$ be a complete separable metric space.
For any $f,f' \in \Lip_1(X)$ and for any Borel probability measures $\mu,\nu$ on $X$, we have
\[
|\, \dKF^\mu(f,f') - \dKF^{\nu}(f,f')\,| \le 2 \dP(\mu,\nu).
\]
\end{lem}

\begin{lem} \label{lem:dconc-pi}
Let $(X,G), (Y,H) \in \cX_\eq$.
\begin{enumerate}
\item For any $\pi,\pi' \in \Pi(X,Y)$ we have
\[
|\,\dconc^\pi((X,G),(Y,H)) - \dconc^{\pi'}(X,G),(Y,H))\,|
\le 2\dP(\pi,\pi').
\]
\item There exists $\pi \in \Pi(X,Y)$ such that
\[
\dconc((X,G),(Y,H)) = \dconc^\pi((X,G),(Y,H)).
\]
\end{enumerate}
\end{lem}

\begin{proof}
(1) follows from Lemma \ref{lem:dKF-dP}.

(2) follows from (1) and the compactness of $\Pi(X,Y)$.
\end{proof}

\begin{lem} \label{lem:fcircg}
Let $X$ be an mm-space.
Then, for any $f,f' \in \Lip_1(X)$ and $g,g' \in \Aut(X)$, we have
\[
\dKF^{\mu_X}(f\circ g,f'\circ g')
\le \dKF^{\mu_X}(f,f') + \dKF^{\mu_X}(g,g').
\]
\end{lem}

\begin{proof}
The lemma follows from
\begin{align*}
\dKF^{\mu_X}(f\circ g,f'\circ g')
&\le \dKF^{\mu_X}(f\circ g,f\circ g')
+ \dKF^{\mu_X}(f\circ g',f'\circ g')\\
&\le \dKF^{\mu_X}(f,f') + \dKF^{\mu_X}(g,g').
\end{align*}
\end{proof}

\begin{lem} \label{lem:dconc-nondeg}
$\dconc$ is nondegenerate on $\cX_\eq$.
\end{lem}

\begin{proof}
Assume $\dconc((X,G),(Y,H)) = 0$ for two elements
$(X,G), (Y,H) \in \cX_\eq$.
By Lemma \ref{lem:dconc-pi},
there is
$\pi \in \Pi(X,Y)$ satisfying $\dconc^\pi((X,G),(Y,H)) = 0$.
Take any $g \in G$.
Then there is a sequence $h_n \in H$, $n=1,2,\dots$,
with $\rho^\pi(g,h_n) \to 0$.
For any $f \in \Lip_1(X)$ there is a sequence
$f_n' \in \Lip_1(Y)$, $n=1,2,\dots$, such that
$\dKF^\pi(f\circ p_1,f_n'\circ p_2) \to 0$ and
$\dKF^\pi(f\circ g\circ p_1,f_n'\circ h_n\circ p_2) \to 0$
as $n\to\infty$.
By the compactness of $H$, we may assume that
$h_n$ converges in measure to $h$ as $n\to\infty$
by replacing $\{h_n\}$ with a subsequence.
Also, we may assume that $f_n'$ converges in measure
to a function $f' \in \Lip_1(Y)$ (see the proof of
\cite{Sy:mmg}*{Lemma 4.45}).
Since
\begin{align*}
& |\dKF^\pi(f\circ p_1,f_n'\circ p_2)-\dKF^\pi(f\circ p_1,f'\circ p_2)|\\
&\le \dKF^\pi(f_n'\circ p_2,f'\circ p_2) = \dKF^{\mu_Y}(f_n',f')
\to 0 \quad\text{as}\ n\to\infty,
\end{align*}
we see that
$\dKF^\pi(f\circ p_1,f'\circ p_2) = 0$, i.e.,
$f\circ p_1 = f'\circ p_2$ on $\supp\pi$.
By using Lemma \ref{lem:fcircg}, the same discussion
yields that
$f\circ g\circ p_1 = f'\circ h\circ p_2$ on $\supp\pi$.

Take any $(x,y) \in \supp\pi$ and fix it.
Applying the above for $f := d_X(x,\cdot)$ yields
that there is $f' \in \Lip_1(Y)$ such that
$f\circ p_1 = f'\circ p_2$ on $\supp\pi$.
For any $(x',y') \in \supp\pi$, we have
$d_X(x,x') = f\circ p_1(x',y') = f'\circ p_2(x',y') = f'(y')$
and in particular, $f'(y) = 0$.
Therefore, $d_X(x,x') = f'(y') = f'(y') - f'(y) \le d_Y(y,y')$.
Exchanging $x$ and $y$ yields the opposite inequality.
We then obtain $d_X(x,x') = d_Y(y,y')$.
This determines an isometry
$\varphi : \Dom(\supp\pi) \to \Imag(\supp\pi)$
satisfying $(x,\varphi(x)) \in \supp\pi$
for any $x \in \Dom(\supp\pi)$.
Since $\Dom(\supp\pi)$ is dense in $\supp\mu_X = X$,
the map $\varphi$ extends to an isometry from $X$ to $Y$,
which we denote by the same notation $\varphi$.
We see that
$\pi = (\id_X,\varphi)_*\mu_X$ and $\varphi_*\mu_X = \mu_Y$,
so that $\varphi$ is an mm-isomorphism.
We also see that $f' = f\circ \varphi^{-1}$
for any $f \in \Lip_1(X)$.
Since $f\circ g\circ p_1 = f\circ \varphi^{-1}\circ h\circ p_2$ on $\supp\pi$, we have
$f\circ g = f\circ \varphi^{-1}\circ h \circ\varphi$.
Take any $x \in X$ and set $f := d_X(g(x),\cdot)$.
Then,
since $0 = f\circ g(x) = f\circ \varphi^{-1}\circ h \circ\varphi(x)$,
we have
$\varphi^{-1}\circ h \circ\varphi(x) = g(x)$ and so
$h = \varphi\circ g \circ\varphi^{-1}$.
This implies $\varphi\circ G \circ\varphi^{-1} \subset H$.
In the same way, $\varphi^{-1}\circ H \circ\varphi \subset G$.
$\rho(g) := \varphi\circ g \circ\varphi^{-1}$
is a group-isomorphism from $G$ to $H$.
Thus, $(X,G)$ and $(Y,H)$ are equivariantly mm-isomorphic
to each other.
This completes the proof.
\end{proof}

\begin{thm} \label{thm:dconc}
$\dconc$ is a metric on $\cX_\eq$.
\end{thm}

\begin{proof}
The theorem follows from
Lemmas \ref{lem:dconc-triangle}(3) and \ref{lem:dconc-nondeg}.
\end{proof}

We next prove the nondegeneracy of $\square$.

For $S \subset X \times Y$ and $f \in \Lip_1(X)$,
we define a function $\tilde{f}_S : Y \to \R$ by
\[
\tilde{f}_S(y) := \inf_{(x',y') \in S} (d_Y(y,y') + f(x')),
\qquad y \in Y,
\]
which is a variant of the McShane-Whitney extension.
It follows from a standard discussion that
$\tilde{f}_S$ is $1$-Lipschitz continuous on $Y$.

\begin{lem} \label{lem:dconc-box}
Let $\pi \in \Pi(X,Y)$ and let
$S,S',S'' \subset \supp\pi$ be three closed subsets.
Then, for any $g,g' \in G$ and $h,h' \in H$, we have
\[
\rho^\pi(g,h)
\le 2 \max\{1-\pi(S),1-\pi(S'),1-\pi(S''),
d^S(g,h),d^{S'}(g',\id_Y),d^{S''}(\id_X,h')\}.
\]
\end{lem}

\begin{proof}
We put
\[
\varepsilon := \max\{1-\pi(S),1-\pi(S'),1-\pi(S''),
d^S(g,h),d^{S'}(g',\id_Y),d^{S''}(\id_X,h')\}.
\]
For any $f \in \Lip_1(X)$, we set $f' := \tilde{f}_{S \cap S'}$.

Let us prove
\begin{equation} \label{eq:dconc-box-1}
\dKF^\pi(f\circ p_1,f'\circ p_2) \le 2\varepsilon. 
\end{equation}
Take any $(x,y) \in S \cap S'$.
It is obvious that $f'(y) \le f(x)$.
From $d^{S'}(g',\id_Y) \le \varepsilon$, we see that
\begin{align*}
f(x)-f'(y) &= \sup_{(x',y')\in S \cap S'} (f(x) - f(x') - d_Y(y,y'))\\
&\le \sup_{(x',y')\in S \cap S'} (d_X(x,x')- d_Y(y,y')) \le 2\varepsilon
\end{align*}
(see the proof of Proposition \ref{prop:box-boxeq}
for the last inequality).
This together with $\pi(S \cap S') \ge 1-2\varepsilon$
implies \eqref{eq:dconc-box-1}.

We next prove
\begin{equation} \label{eq:dconc-box-2}
\dKF^\pi(f\circ g\circ p_1,f'\circ h\circ p_2) \le 4\varepsilon.
\end{equation}
Take any $(x,y) \in S \cap S' \cap (g^{-1}(\Dom(S \cap S')) \times Y)$.
Since $d^S(g,h) \le \varepsilon$, we see
\begin{align*}
f'\circ h(y) &= \inf_{(x',y') \in S \cap S'} (d_Y(h(y),y') + f(x'))\\
&\le \inf_{(x',y') \in S \cap S'} (d_X(g(x),x') + f(x')) + \varepsilon
\le f\circ g(x)+\varepsilon.
\end{align*}
By the $1$-Lipschitz continuity of $f$ and
by $d^S(g,h) \le \varepsilon$,
\begin{align*}
f\circ g(x)-f'\circ h(y) &= \sup_{(x',y')\in S \cap S'} (f\circ g(x) - f(x') - d_Y(h(y),y'))\\
&\le \sup_{(x',y')\in S \cap S'} (d_X(g(x),x')- d_Y(h(y),y')) \le \varepsilon.
\end{align*}
We thus obtain
$|f\circ g(x)-f'\circ h(y)| \le \varepsilon$, which together with
$\pi(S \cap S' \cap (g^{-1}(\Dom(S \cap S')) \times Y))
\ge 1-4\varepsilon$
implies \eqref{eq:dconc-box-2}.

\eqref{eq:dconc-box-1} and \eqref{eq:dconc-box-2}
together imply
$\rho^\pi_{g,h}(f,f') \le 4\varepsilon$.

In the same way, 
for any $f' \in \Lip_1(Y)$ there is $f \in \Lip_1(X)$
such that $\rho^\pi_{g,h}(f,f') \le 4\varepsilon$.
We thus obtain $\rho^\pi(g,h) \le 4\varepsilon$.
This completes the proof.
\end{proof}

\begin{prop} \label{prop:dconc-box}
For any $(X,G),(Y,H) \in \cX_\eq$ we have
\[
\dconc((X,G),(Y,H)) \le 4\,\square((X,G),(Y,H)).
\]
\end{prop}

\begin{proof}
Take any $\varepsilon$ with
$\varepsilon > \square((X,G),(Y,H))$.
There is $\pi \in \Pi(X,Y)$ such that
$\square^\pi((X,G),(Y,H)) < \varepsilon$.
For any $g \in G$ there are $h \in H$ and
a closed subset $S \subset \supp\pi$ such that
$\pi(S) \ge 1-\varepsilon$ and $d^S(g,h) \le \varepsilon$.
This together with Lemma \ref{lem:dconc-box} implies
$\rho^\pi(g,h) \le 4\varepsilon$.
The same is true if we exchange $g$ and $h$.
We thus obtain
$\dconc((X,G),(Y,H)) \le \dconc^\pi((X,G),(Y,H)) < 4\varepsilon$. 
This completes the proof.
\end{proof}

\begin{thm} \label{thm:box-dist}
$\square$ is a metric on $\cX_\eq$.
\end{thm}

\begin{proof}
Lemma \ref{lem:box-triangle} says a triangle inequality for $\square$.
The symmetry is trivial.
Combining Lemma \ref{lem:dconc-nondeg}
and Proposition \ref{prop:dconc-box} implies
the nondegeneracy of $\square$.
This completes the proof.
\end{proof}

Lemma \ref{lem:dconc-nondeg}
and Theorem \ref{thm:box-dist} also lead us
the nondegeneracy of the original observable and box metrics
without group actions.   This gives a new proof
of the nondegeneracy.

\section{Convergence of quotient spaces}

In this section,
we prove that if a sequence of mm-spaces with group actions
converges with respect to $\square/\dconc$, then so does the sequence of the quotient spaces.

For $(X,G) \in \cX_\eq$, we denote by
$d_{X/G}$ the \emph{quotient metric} on $X/G$
(which is defined as the distance between orbits)
and set $\mu_{X/G} := (\pi_G)_*\mu_X$,
where $\pi_G : X \to X/G$ is the quotient map.
Note that $d_{X/G}$ is a pseudo-metric in general.

\begin{lem}
For any $(X,G) \in \cX_\eq$, any $G$-orbit is closed
and in particular, $d_{X/G}$ is a metric.
\end{lem}

\begin{proof}
Let $y \in X$ be a point in
the closure of the $G$-orbit of a point $x \in X$.
It then suffices to prove that $y$ belongs to the $G$-orbit of $x$.
There is a sequence $\{g_n\} \subset G$ such that
$g_n(x) \to y$ as $n\to\infty$.
By Lemma \ref{lem:Aut-conv}, replacing $\{g_n\}$ with
a subsequence, we see that
$g_n$ converges pointwise to an element $g \in \Aut(X)$.
Hence, $g_n(x) \to g(x)$ and $y = g(x)$.
Since $G$ is closed, $g$ belongs to $G$.
This completes the proof.
\end{proof}

We consider the thick-thin decomposition.
Assume that a sequence of mm-spaces $X_n$, $n=1,2,\dots$,
box-converges to an mm-space $Y$.
For $v, r > 0$, we put
$Y(v,r) := \{\,y \in Y \mid \mu_Y(B_r(y)) > v\,\}$.
For any $\varepsilon > 0$ there is $v_\varepsilon > 0$
such that
$\mu_Y(Y(v_\varepsilon,\varepsilon)) > 1 - \varepsilon$
and $v_\varepsilon \to 0$ as $\varepsilon \to 0+$.
Let $Y_\varepsilon := Y(v_\varepsilon,\varepsilon)$ and
$X_{n,\varepsilon} := \{\,x \in X_n \mid \mu_{X_n}(B_{2\varepsilon}(x)) > v_\varepsilon/2\,\}$,
which we call the \emph{thick parts} of $Y$ and $X_n$ 
respectively.
Then, $X_{n,\varepsilon}$ is $\Aut(X_n)$-invariant
and $Y_\varepsilon$ is $\Aut(Y)$-invariant.

\begin{lem} \label{lem:Xneps}
For any $n$ and $\varepsilon > 0$ with
$\square(X_n,Y) < v_\varepsilon/2$, we have
\[
\mu_{X_n}(X_{n,\varepsilon}) > 1-\varepsilon - \square(X_n,Y).
\]
\end{lem}

\begin{proof}
Let $\varepsilon_n := \square(X_n,Y)$.
There is $\pi_n \in \Pi(X_n,Y)$ for each $n$
such that $\square^{\pi_n}(X_n,Y) = \varepsilon_n$.
Take any $x \in S_n^{-1}(Y_\varepsilon)$.
There is $y \in Y_\varepsilon$ with $(x,y) \in S_n$.
Since
$p_1((X_n \times B_\varepsilon(y)) \cap S_n) = S_n^{-1}(B_\varepsilon(y)) \subset B_{2\varepsilon}(x)$, we see
\begin{align*}
&\mu_{X_n}(B_{2\varepsilon}(x)) \ge \mu_{X_n}(p_1((X_n \times B_\varepsilon(y)) \cap S_n))\\
& = \pi_n(p_1^{-1}(p_1((X_n \times B_\varepsilon(y)) \cap S_n))
\ge \pi_n((X_n \times B_\varepsilon(y)) \cap S_n)\\
&\ge v_\varepsilon-\varepsilon_n > v_\varepsilon/2,
\end{align*}
which implies $x \in X_{n,\varepsilon}$.
Therefore, $S_n^{-1}(Y_\varepsilon) \subset X_{n,\varepsilon}$
and $\mu_{X_n}(X_{n,\varepsilon}) \ge \mu_{X_n}(S_n^{-1}(Y_\varepsilon))\\
\ge \pi_n((X_n \times Y_\varepsilon) \cap S_n)
> 1-\varepsilon-\varepsilon_n$.
This completes the proof.
\end{proof}

\begin{lem} \label{lem:Gn-precpt}
If a sequence $(X_n,G_n) \in \cX_\eq$, $n=1,2,\dots$,
equivariantly box converges to an element $(Y,H) \in \cX_\eq$,
then $\{(G_n,\dKF)\}_{n=1}^\infty$ is precompact
with respect to the Gromov-Hausdorff topology.
\end{lem}

\begin{proof}
Take any $\delta > 0$.
Let $\{g_i\}_{i=1}^N$ be a $\delta$-discrete net of $G_n$
with respect to $\dKF$.
It suffices to prove that $N$ is bounded from above for every $n$.
There is $\varepsilon > 0$ such that
$v_\varepsilon/2 + 4\varepsilon < \delta$.
Assume that $n$ is so large that
$\square((X_n,G_n),(Y,H)) < \varepsilon/2$.

Let us prove $d^{\pi_n}(g_i,g_j) \ge v_\varepsilon/2$
for all $i \neq j$.
If $d^{\pi_n}(g_i,g_j) < v_\varepsilon/2$
for a pair of distinct $i$ and $j$,
then there is a closed subset $S_n \subset \supp\pi_n$
such that $\pi_n(S_n) \ge 1-v_\varepsilon/2$ and
\begin{equation} \label{eq:Gn-precpt}
|\,d_{X_n}(g_i(x_1),x_2)  - d_{X_n}(g_j(x_1),x_2)\,| \le v_\varepsilon/2
\end{equation}
for any $x_1,x_2 \in p_1(S_n)$.
Put $S_{n,\varepsilon} := (X_{n,\varepsilon} \times Y) \cap S_n$.
Since Proposition \ref{prop:box-boxeq} implies
$\square(X_n,Y) < \varepsilon$, we see
by Lemma \ref{lem:Xneps} that
$\mu_{X_n}(X_{n,\varepsilon}) \ge 1-2\varepsilon$,
which proves
$\pi_n(S_{n,\varepsilon}) \ge 1 - v_\varepsilon/2 - 2\varepsilon$.
We also have
$\mu_{X_n}(p_1(S_n)) \ge (p_1)_*\pi_n(p_1(S_n))
\ge \pi_n(S_n) \ge 1-v_\varepsilon/2$.
For any $x \in p_1(S_{n,\varepsilon})$,
we see $g_i(x) \in X_{n,\varepsilon}$
and so there is $x' \in p_1(S_n)$ such that
$d_{X_n}(g_i(x),x') \le 2\varepsilon$.
Since
\eqref{eq:Gn-precpt} implies
$|\,d_{X_n}(g_i(x),x')  - d_{X_n}(g_j(x),x')\,| \le v_\varepsilon/2$,
we have $d_{X_n}(g_j(x),x') \le v_\varepsilon/2 + 2\varepsilon$.
By a triangle inequality,
$d_{X_n}(g_i(x),g_j(x)) \le v_\varepsilon/2 + 4\varepsilon < \delta$.
This contradicts $\dKF(g_i,g_j) \ge \delta$.
We have proved $d^{\pi_n}(g_i,g_j) \ge v_\varepsilon/2$
for all $i \neq j$.

The net $\{g_i\}_{i=1}^N$ is $(v_\epsilon/2)$-discrete
with respect to $d^{\pi_n}$.
It follows from Lemma \ref{lem:Aut-conv} that
$G_n$ and $H$ are compact with respect to $d^{\pi_n}$,
where we remark that $d^{\pi_n}$ on $H$ is independent of $n$.
By $\square^{\pi_n}((X_n,G_n),(Y,H)) \to 0$,
we have the precompactness of $\{G_n\}$ with respect to 
$d^{\pi_n}$.  Thus, $N$ is bounded from above.
This completes the proof.
\end{proof}

\begin{thm} \label{thm:box-quot}
If a sequence $(X_n,G_n) \in \cX_\eq$, $n=1,2,\dots$,
equivariantly box converges to an element $(Y,H) \in \cX_\eq$,
then $X_n/G_n$ box converges to $Y/H$.
\end{thm}

\begin{proof}
Let $\varepsilon_n := \square((X_n,G_n),(Y,H))$.
There is $\pi_n \in \Pi(X_n,Y)$ for each $n$ such that
$\square^{\pi_n}((X_n,G_n),(Y,H)) = \varepsilon_n$.

Taking any $\varepsilon > 0$, we set
$\varepsilon' := \min\{v_\varepsilon,\varepsilon\}/3$.
Assume that $n$ is so large that
$\varepsilon_n$ is very small compared with $\varepsilon'$.
By the precompactness of $\{G_n\}$
(Lemma \ref{lem:Gn-precpt}) and by the compactness of $H$,
there are an $\varepsilon'$-net $\{g_{n,k}\}_{k=1}^N$ of $G_n$
and an $\varepsilon'$-net $\{h_{n,k}\}_{i=1}^N$ of $H$
with respect to the Ky Fan metric such that
$d^{\pi_n}(g_{n,k},h_{n,k}) \le \varepsilon_n$,
where $N$ is independent of $n$.
There is a closed subset $S_{k,n} \subset \supp\pi_n$
for each $k$ such that
$\pi_n(S_{k,n}) \ge 1-\varepsilon_n$ and,
for any $(x_i,y_i) \in S_{k,n}$,
\[
|\,d_{X_n}(g_{n,k}(x_1),x_2) - d_Y(h_{n,k}(y_1),y_2))\,| \le \varepsilon_n.
\]
Put $S_n := (X_{n,\varepsilon} \times Y_\varepsilon)
\cap \bigcap_{k=1}^N S_{k,n}$.
Define a map
$p_{G_n} \times p_H : X_n \times Y \to X_n/G_n \times Y/H$
by
$p_{G_n} \times p_H(x,y) := (p_{G_n}(x),p_H(y))$,
and set
$\sigma_n := (p_{G_n} \times p_H)_*\pi_n$ and
$\mathcal{S}_n := p_{G_n} \times p_H(S_n)$.
We have
$\sigma_n(\mathcal{S}_n)
= \pi_n((p_{G_n}\times p_H)^{-1}(p_{G_n}\times p_H(S_n)))
\ge \pi_n(S_n) \ge 1-2\varepsilon-N\varepsilon_n \ge 1-3\varepsilon$.
Take any $([x_i],[y_i]) \in \mathcal{S}_n$, $i=1,2$.
We may assume $(x_i,y_i) \in S_n$.
There is $h \in H$ with $d_Y(h(y_1),y_2) = d_{Y/H}([y_1],[y_2])$.
Also there is $k$ such that $\dKF(h,h_{n,k}) \le \varepsilon'$.
It follows from $y_1 \in Y_\varepsilon$ that
$\mu_Y(B_\varepsilon(y_1)) \ge v_\varepsilon > \varepsilon'$,
so that there is $y_1' \in B_\varepsilon(y_1)$
with $d_Y(h_{n,k}(y_1'),h(y_1')) \le \varepsilon'$.
By the isometric property of $h_{n,k}$, $h$
and by a triangle inequality, we have
$d_Y(h_{n,k}(y_1),h(y_1)) \le 2\varepsilon + \varepsilon' \le 3\varepsilon$ and hence
\begin{align*}
& d_{X_n/G_n}([x_1],[x_2]) \le d_{X_n}(g_{n,k}(x_1),x_2)
\le d_Y(h_{n,k}(y_1),y_2) + \varepsilon_n\\
&\le  d_Y(h(y_1),y_2) + 3\varepsilon + \varepsilon_n
\le d_{Y/H}([y_1],[y_2]) + 4\varepsilon.
\end{align*}

In the same way, there is $g \in G_n$ with
$d_Y(g(x_1),x_2) = d_{X_n/G_n}([x_1],[x_2])$.
There is $k$ with $\dKF(g,g_{n,k}) \le \varepsilon'$.
It follows from $x_1 \in X_{n,\varepsilon}$ that
$\mu_{X_n}(B_{2\varepsilon}(x_1)) \ge v_\varepsilon/2 >
\varepsilon'$, so that there is
$x_1' \in B_{2\varepsilon}(x_1)$ with
$d_{X_n}(g_{n,k}(x_1'),g(x_1')) \le \varepsilon'$.
We have $d_{X_n}(g_{n,k}(x_1),g(x_1)) \le 4\varepsilon + \varepsilon' \le 5\varepsilon$ and hence
\[
d_{Y/H}([y_1],[y_2]) \le d_Y(h_{n,k}(y_1),y_2)
\le d_{X_n/G_n}([x_1],[x_2]) + 5\varepsilon.
\]
Thus, for any $([x_i],[y_i]) \in \mathcal{S}_n$, $i=1,2$,
\[
|\,d_{X_n/G_n}([x_1],[x_2]) - d_{Y/H}([y_1],[y_2])\,| \le 5\varepsilon,
\]
which implies $\square(X_n/G_n,Y/H) \le 5\varepsilon$.
This completes the proof.
\end{proof}

\begin{lem} \label{lem:dKF-L1}
Let $X$ be an mm-space, $f, g: X \to \R$ two $L^1$-functions,
and let $D \ge 0$.  Then we have the following.
\begin{enumerate}
\item $\dKF(f,g)^2 \le \|f - g\|_{L^1}$.
\item If $|f|, |g| \le D$, then
$\|f-g\|_{L^1} \le (2D+1) \dKF(f,g)$.
\end{enumerate}
\end{lem}

\begin{proof}
Put $\varepsilon := \dKF(f,g)$.

We prove (1).
Since $\mu_X(|f-g| \ge \varepsilon) \ge \varepsilon$, we have
\begin{align*}
\|f - g\|_{L_1} \ge \int_{\{|f-g| \ge \varepsilon\}} |f-g| d\mu_X
\ge \varepsilon \mu_X(|f-g| \ge \varepsilon) \ge \varepsilon^2
\end{align*}

We prove (2).  If $|f|, |g| \le D$, then
\begin{align*}
\|f - g\|_{L_1}
&\le \int_{\{|f-g| > \varepsilon\}} |f-g| \;d\mu_X
+\int_{\{|f-g| \le \varepsilon\}} |f-g| \;d\mu_X\\
&\le 2D \mu_X(|f-g| > \varepsilon) + \varepsilon
\le (2D+1)\varepsilon.
\end{align*}
This completes the proof.
\end{proof}

\begin{defn}[Observable diameter \cite{Gmv:green}]
For $\kappa > 0$, we define
the \emph{$\kappa$-observable diameter} of an mm-space $X$
to be
\[
\ObsDiam(X;-\kappa) := \sup_{f \in \Lip_1(X)} \diam(f_*\mu_X;-\kappa),
\]
where $\diam(\nu;-\kappa)$ for a Borel measure $\nu$ on $\R$
is the infimum of the diameter of Borel subsets $A \subset \R$
with $\nu(A) \ge 1-\kappa$.
\end{defn}

\begin{thm} \label{thm:conc-quot}
For any $(X,G), (Y,H) \in \cX_\eq$ and $\kappa > 0$,
setting
\[
D_\kappa := \max\{\ObsDiam(X;-\kappa),\ObsDiam(Y;-\kappa)\},
\]
we have
\[
\dconc(X/G,Y/H) \le \sqrt{3(D_\kappa+1)\dconc((X,G),(Y,H))}
+\kappa.
\]
In particular, if a sequence $(X_n,G_n) \in \cX_\eq$, $n=1,2,\dots$,
$\dconc$-converges to an element $(Y,H) \in \cX_\eq$,
then $X_n/G_n$ $\dconc$-converges to $Y/H$.
\end{thm}

\begin{proof}
We first remark that we have the one to one
correspondence between
the set of $G$-invariant functions in $\Lip_1(X)$
and $\Lip_1(X/G)$, where
$f(x) = \bar{f}([x])$ for $G$-invariant $f \in \Lip_1(X)$
and $\bar{f} \in \Lip_1(X/G)$.

Take any $\varepsilon$ with $\dconc((X,G),(Y,H)) < \varepsilon$.
There is $\pi \in \Pi(X,Y)$ such that
$\dconc^\pi((X,G),(Y,H)) < \varepsilon$.
For any $h \in H$ there is $g \in G$ with
$\rho^\pi(g,h) < \varepsilon$.
We take any $G$-invariant function $\hat{f} \in \Lip_1(X)$.
Since $\diam(\hat{f}_*\mu_X;-\kappa) \le D_\kappa$,
there is a real number $c$ such that
$\mu_X(|\hat{f}+c| \le D_\kappa/2) \ge 1-\kappa$.
Setting
$f(x) := \min\{\max\{\hat{f}(x)+c,-D_\kappa/2\},D_\kappa/2\}$,
$x \in X$, we see that
$|f| \le D_\kappa/2$ and $\dKF(f,\hat{f}+c) \le \kappa$.
By $\rho^\pi(g,h) < \varepsilon$,
there is $f_h' \in \Lip_1(Y)$ with $|f_h'| \le D_\kappa/2$
such that $\rho^\pi_{g,h}(f,f_h') < \varepsilon$.
Hence,
\[
\dKF^\pi(f\circ p_1,f_h'\circ p_2) < \varepsilon
\quad\text{and}\quad
\dKF^\pi(f\circ p_1,f_h'\circ h\circ p_2) < \varepsilon.
\]
Setting $f' := f_{\id_Y}'$ yields
$\dKF^\pi(f\circ p_1,f'\circ p_2) < \varepsilon$.
By a triangle inequality,
$\dKF^{\mu_Y}(f'\circ h,f_h'\circ h) = \dKF^{\mu_Y}(f',f_h') = \dKF^\pi(f'\circ p_2,f_h'\circ p_2) < 2\varepsilon$.
By a triangle inequality again,
$\dKF^\pi(f\circ p_1,f'\circ h\circ p_2) < 3\varepsilon$.
Denote by $\mu_{\Aut(Y)}$ the Haar probability measure
on $\Aut(Y)$ and define
\[
f''(y) := \int_{\Aut(Y)} f'\circ h(y) \, d\mu_{\Aut(Y)}(h),
\qquad y \in Y.
\]
The $f''$ is an $H$-invariant $1$-Lipschitz function
and satisfies $|f''| \le D_\kappa/2$.
For any $(x,y) \in \supp\pi$,
\begin{align*}
|f(x) - f''(y)|
&= \left|\int_{\Aut(Y)} (f(x) - f'\circ h(y)) \;d\mu_{\Aut(Y)}(h)\right|\\
&\le \int_{\Aut(Y)} |f(x) - f'\circ h(y)| \;d\mu_{\Aut(Y)}(h).
\end{align*}
Integrate this formula with respect to $d\pi(x,y)$
and apply Fubini's theorem and Lemma \ref{lem:dKF-L1}(2)
to obtain
\begin{align*}
\|f\circ p_1-f''\circ p_2\|_{L_1} &\le \int_{\Aut(Y)} \|f\circ p_1 - f'\circ h \circ p_2\|_{L_1} d\mu_{\Aut(Y)}(h)
\le 3(D_\kappa+1) \varepsilon,
\end{align*}
which together with Lemma \ref{lem:dKF-L1}(1)
implies
$\dKF^\pi(f\circ p_1,f''\circ p_2) \le \sqrt{3(D_\kappa+1) \varepsilon}$.
Therefore, $\dKF^\pi(\hat{f}\circ p_1 + c,f''\circ p_2)
\le \sqrt{3(D_\kappa+1) \varepsilon} + \kappa$.
The same is true if we exchange $X$ and $Y$.
We thus obtain the first part of the theorem.

The latter follows from the first
because of the boundedness of $\ObsDiam(X_n;-\kappa)$
which is implied by \cite{OzSy:limf}*{Theorem 1.1}.
\end{proof}

\section{Convergence of group actions}

In this section, we prove that the projection
$\cX_\eq \ni (X,G) \mapsto X \in \cX$
is a proper map with respect to $\square$,
which is analogous to \cite{FY:f-gp}*{Proposition 3.6},
where $\cX$ denotes the set of mm-isomorphism classes of mm-spaces.

We need some definition and lemmas for the proof.

\begin{defn}[Weak Hausdorff convergence]
Let $X$ be a metric space and
$S,S_n \subset X$, $n=1,2,\dots$, closed subsets.
We say that $S_n$ converges to $S$ in the \emph{weak Hausdorff}
sense if the following (i) and (ii) are satisfied.
\begin{enumerate}
\item[(i)] $\lim_{n\to\infty} d_X(x,S_n) = 0$ for any $x \in S$.
\item[(ii)] $\liminf_{n\to\infty} d_X(x,S_n) > 0$
for any $x \in X \setminus S$.
\end{enumerate}
\end{defn}

Note that $S_n$ and $S$ may be empty and
we set $d_X(x,\emptyset) = \infty$.

\begin{lem}[\cite{Sy:mmg}*{Lemma 6.6}]\label{lem:weakH-cpt}
The set of closed subsets of a metric space
is sequentially compact with respect to
the weak Hausdorff convergence.
\end{lem}

\begin{lem} \label{lem:wHaus-cpt}
Let $X$ be a metric space and
$S,S_n \subset X$, $n=1,2,\dots$, closed subsets.
If $S_n$ converges to $S$ in the weak Hausdorff sense,
then for any compact subset $K \subset X$ and for any
$\varepsilon > 0$, there is a number $n_0$ such that
$S_n \cap K \subset U_\varepsilon(S \cap K)$ for $n \ge n_0$.
\end{lem}

Lemma \ref{lem:wHaus-cpt} follows from an easy discussion
using the compactness of $K$ and the proof is omitted.

\begin{lem} \label{lem:lim-supp}
Let $X$ be a complete separable metric space,
$S,S_n \subset X$ closed subsets,
and $\mu, \mu_n$ Borel probability measures, $n=1,2,\dots$.
If $\mu_n$ converges weakly to $\mu$
and if $S_n$ converges to $S$ in the weak Hausdorff,
then
\[
\mu(S) \ge \limsup_{n\to\infty} \mu_n(S_n).
\]
\end{lem}

\begin{proof}
Since Prokhorov's theorem implies the tightness of $\{\mu_n\}$,
there is a sequence of compact subsets
$K_m \subset X$, $m=1,2,\dots$, such that
$\mu_n(K_m) \ge 1-1/m$ and $\mu(K_m) \ge 1-1/m$
for every $m$.
It follows from Lemma \ref{lem:wHaus-cpt} that,
for any $\varepsilon > 0$,
\[
\mu(U_\varepsilon(S \cap K_m))
\ge \limsup_{n\to\infty} \mu(S_n \cap K_m)
\ge \limsup_{n\to\infty} \mu(S_n) - 1/m,
\]
which implies
\[
\mu(S) \ge \mu(S \cap K_m) \ge \limsup_{n\to\infty} \mu(S_n) - 1/m.
\]
This proves the lemma.
\end{proof}

The following is a main theorem of this section.

\begin{thm} \label{thm:subconv}
The projection $\cX_\eq \ni (X,G) \mapsto X \in \cX$
is a proper map with respect to $\square$, i.e.,
for given $(X_n,G_n) \in \cX_\eq$, $n=1,2,\dots$ and
$Y \in \cX$,
if $X_n$ box converges to $Y$ as $n\to\infty$, then
there exists a closed subgroup $H \subset \Aut(X)$
such that 
$(X_n,G_n)$ equivariantly box converges to
$(Y,H)$.
\end{thm}

\begin{proof}
By $\square(X_n,Y) \to 0$,
there are $\pi_n \in \Pi(X_n,Y)$,
closed subsets $S_n \subset X_n \times Y$,
and $\varepsilon_n \to 0+$ such that
$\pi_n(S_n) \ge 1-\varepsilon_n$ and
\[
\sup_{(x_i,y_i) \in S_n} |\,d_{X_n}(x_1,x_2)-d_Y(y_1,y_2)\,|
\le \varepsilon_n.
\]

For any $g \in G_n$ we put
$\rho_n(g) := S_n\circ g\circ S_n^{-1} \subset Y \times Y$
and
$\pi_n^g := \pi_n \circ (\id_{X_n},g)_*\mu_{X_n} \circ \pi_n^{-1}
\in \Pi(Y,Y)$, where we define $tr(y,y') := (y',y)$ for $y,y' \in Y$
and $\pi_n^{-1} := tr_*\pi_n$.
It follows from the $\dP$-compactness of $\Pi(Y,Y)$ that
if we take a subsequence of $\{n\}$, then
$\{\,\pi_n^g \mid g \in G_n\,\}$ converges to
a closed subset $\Pi \subset \Pi(Y,Y)$ in the Hausdorff sense.
Take any $\pi \in \Pi$.
There is a sequence $g_n \in G_n$ with $\pi_n^{g_n} \to \pi$.
Taking a subsequence, we see that
$\rho_n(g_n)$ converges to
a closed subset $\hat{h} \subset Y \times Y$ in the weak Hausdorff sense.  By Lemma \ref{lem:lim-supp}, we have
$\pi(\hat{h}) \ge \lim_{n\to\infty} \pi_n^{g_n}(\rho_n(g_n)) = 1$
and so $\hat{h} \supset \supp\pi$.
For any $y_i \in \Dom(\rho_n(g_n))$
and $y_i' \in \rho_n(g_n)(y_i)$, $i=1,2$,
there are $x_1,x_2 \in X_n$ such that $(x_i,y_i) \in S_n$ and
$(g_n(x_i),y_i') \in S_n$ for $i=1,2$.  Therefore,
\begin{align*}
|\,d_{X_n}(x_1,x_2) - d_Y(y_1,y_2)\,| \le \varepsilon_n
\ \text{and}\ 
|\,d_{X_n}(g_n(x_1),g_n(x_2)) - d_Y(y_1',y_2')\,| \le \varepsilon_n.
\end{align*}
Since $g_n$ is an isometry, these inequalities together imply
\begin{equation} \label{eq:subconv}
|\,d_Y(y_1,y_2) - d_Y(y_1',y_2')\,| \le 2\varepsilon_n.
\end{equation}

Taking a subsequence we assume that
$\varepsilon_n \le 1/2^n$ for every $n$.\\
Since $\mu_Y(\Dom(\rho_n(g_n))) \ge 1-\pi_n(S_n)-\pi_n^{-1}(S_n^{-1})
\ge 1-2\varepsilon_n$, setting
\[
D := \bigcup_{m=1}^\infty \bigcap_{n=m}^\infty \Dom(\rho_n(g_n))
\]
we see that
\[
\mu_Y(Y \setminus D)
\le \mu_Y\left( \bigcup_{n=m}^\infty (Y \setminus \Dom(\rho_n(g_n))\right) \le 2\sum_{n=m}^\infty \varepsilon_n
= \frac{4}{2^m} \to 0 \ \text{as}\ m\to\infty,
\]
which implies $\mu_Y(D) = 1$.
For any fixed $y_1,y_2 \in D$, since \eqref{eq:subconv}
holds for infinitely many $n$, we have
$d_Y(y_1,y_2) = d_Y(y_1',y_2')$.
Moreover, by the closedness of $\hat{h}$,
the set $\hat{h}$ is the graph of an isometry, say $h$,
and satisfies $\supp\pi = \hat{h}$.
we have $(\id_Y,h)_*\mu_Y = \pi$,
which implies $h_*\mu_Y = \mu_Y$ and then
$h \in \Aut(Y)$.
Thus, the support of any measure in $\Pi$ is the graph
of an element of $\Aut(Y)$.
Let $H$ be the set of all such maps of $\Aut(Y)$.

Let us prove the closedness of $H$ with respect to $\dKF$.
Assume that a sequence $h_n \in H$ converges in measure
to a map $h \in \Aut(Y)$.
It holds that $\dP((\id_Y,h_n)_*\mu_Y,(\id_Y,h)_*\mu_Y)
\le \dKF((\id_Y,h_n),(\id_Y,h)) = \dKF(h_n,h) \to 0$.
Since $(\id_Y,h_n)_*\mu_Y \in \Pi$
and by the compactness of $\Pi$,
we have $(\id_Y,h)_*\mu_Y \in \Pi$ and so $h \in H$.
This proves the closedness of $H$.

Let us next prove that $H$ is a subgroup of $\Aut(Y)$.
Taking any $h, h' \in H$ we first prove $h'\circ h \in H$.
By taking a subsequence of $\{n\}$,
there are $g_n,g_n' \in G_n$ such that
$\rho_n(g_n)$ and $\rho_n(g_n')$ respectively
converge to the graphs of $h$ and $h'$
in the weak Hausdorff sense.
Taking a subsequence again, we see that
$\rho_n(g_n'\circ g_n)$ converges to the graph of a map $k \in H$
in the weak Hausdorff sense.
It suffices to prove $h' \circ h = k$.
We put $D_n := \Dom(\rho_n(g_n) \circ \rho_n(g_n'))
\cap \Dom(\rho_n(g_n'\circ g_n))$.
For any $x \in D_n$, there are $y,z \in Y$ such that
$(x,y) \in \rho_n(g_n)$ and $(y,z) \in \rho_n(g_n')$.
Also, there are $x',y' \in X_n$ such that
$(x',x),(g_n(x'),y),(y',y),(g_n'(y'),z) \in S_n$.
There is $z' \in Y$ such that
$(x,z') \in \rho_n(g_n'\circ g_n)$,
and there is $x'' \in X_n$ such that
$(x'',x),(g_n'\circ g_n(x''),z') \in S_n$.
Since
$d_{X_n}(g_n'\circ g_n(x'),g_n'(y')) = d_{X_n}(g_n(x'),y') \le \varepsilon_n$
and
$d_{X_n}(g_n'\circ g_n(x'),g_n'\circ g_n(x'')) = d_{X_n}(x',x'')
\le \varepsilon_n$, we see, by a triangle inequality,
$d_{X_n}(g_n'\circ g_n(x''),g_n'(y')) \le 2\varepsilon_n$,
which proves
\begin{equation} \label{eq:subgp}
d_Y(z,z') \le 3\varepsilon_n.
\end{equation}
In the same way as before, taking a subsequence of $\{n\}$,
we may assume that
$D := \bigcup_{m=1}^\infty \bigcap_{n=m}^\infty D_n$
has $\mu_Y$-measure equal to one.
By \eqref{eq:subgp},
we have $h'\circ h(x) = k(x)$ for any $x \in D$,
which implies $h' \circ h = k \in H$.
A similar (simpler) discussion leads us to $h^{-1} \in H$.
Thus, $H$ is a subgroup of $\Aut(Y)$.

Let us prove that $\square((X_n,G_n),(Y,H)) \to 0$.
We set $\delta_n := \dH(\{\,\pi_n^g \mid g \in G_n\,\},\Pi)$,
which is the Hausdorff distance with respect to $\dP$.
Let $g \in G_n$ and $h \in H$ be such that
$\dP(\pi_n^g,\pi^h) \le \delta_n$,
where $\pi^h := (\id_Y,h)_*\mu_Y \in \Pi$.
It follows from Strassen's theorem that
there are a coupling $\omega$ between $\pi_n^g$ and $\pi^h$
and a closed subset $T^{g,h} \subset \supp\omega$ with
$\omega(T^{g,h}) \ge 1-\delta_n$ such that
$d_{Y\times Y}((y,y'),(\eta,\eta')) \le \delta_n$
for any $(y,y',\eta,\eta') \in T^{g,h}$,
where $d_{Y\times Y}$ denotes the $l^2$-product metric of $d_Y$.
Note that $\eta' = h(\eta)$ holds.  We put
\begin{align*}
S_n^{g,h} :=
\{\,(x,y) \in S_n \mid (g(x),y') \in S_n,
\ (y,y',\eta,h(\eta)) \in T^{g,h}\ \text{for some $y',\eta \in Y$}\,\}.
\end{align*}
For any $(x_i,y_i) \in S_n^{g,h}$, $i=1,2$,
there are $y_1',\eta \in Y$ such that
$(g(x_1),y_1') \in S_n$, $(y_1,y_1',\eta,h(\eta)) \in T^{g,h}$.
Since
\begin{align*}
& |\,d_{X_n}(g(x_1),x_2) - d_Y(y_1',y_2)\,| \le \varepsilon_n,\\
& d_Y(y_1',h(y_1)) \le d_Y(y_1',h(\eta)) + d_Y(h(\eta),h(y_1)) \le 2\delta_n,
\end{align*}
we have
\begin{equation} \label{eq:subconv1}
|\,d_{X_n}(g(x_1),x_2) - d_Y(h(y_1),y_2)\,| \le \varepsilon_n + 2\delta_n.
\end{equation}

We are going to estimate $\pi_n(S_n^{g,h})$.
Note that
\begin{align*}
S_n^{g,h} =
\{\,(x,y) \in S_n \mid (x,y') \in S_n \circ g,
\ (y,y') \in p_{12}(T^{g,h})\ \text{for some $y' \in Y$}\,\}.
\end{align*}
Let
\begin{align*}
\tilde{S}_n^{g,h} &:= \{\,(y,x,y') \in Y \times X_n \times Y \mid
(x,y) \in S_n,
\ (x,y') \in S_n \circ g,\ (y,y') \in p_{12}(T^{g,h})\,\}\\
&= (S_n^{-1} \times Y) \cap (Y \times S_n\circ  g)
\cap [p_{12}(T^{g,h}) \times X_n]_{23},\\
\tilde{\pi}_n^g &:= (\pi_n \circ (\id_{X_n},g)_*\mu_{X_n}) \bullet \pi_n^{-1},
\end{align*}
where $[...]_{23}$ means to exchange the second and third
components.
Since
$S_n^{g,h} = p_{21}(\tilde{S}_n^{g,h})$,
$\pi_n(S_n) \ge 1-\varepsilon_n$,
$\pi_n^g(p_{12}(T^{g,h})) \ge \omega(T^{g,h}) \ge 1-\delta_n$,
$(p_{13})_*\tilde{\pi}_n^g = \pi_n^g$,
$(p_{12})_*\tilde{\pi}_n^g = \pi_n^{-1}$, and
$(p_{23})_*\tilde{\pi}_n^g = \pi_n \circ (\id_{X_n},g)_*\mu_X$,
we have
\begin{align} \label{eq:subconv2}
&\pi_n(S_n^{g,h}) = (p_{21})_*\tilde{\pi}_n^g(p_{21}(\tilde{S}_n^{g,h}))
= \tilde{\pi}_n^g(p_{21}^{-1}(p_{21}(\tilde{S}_n^{g,h})))
\ge \tilde{\pi}_n^g(\tilde{S}_n^{g,h})\\
&= 1-\tilde{\pi}_n^g((S_n^{-1} \times Y)^c \cup (Y \times S_n\circ  g)^c \cup [p_{12}(T^{g,h}) \times X_n]_{23}^c) \notag\\
&\ge 1 - (1-\tilde{\pi}_n^g(S_n^{-1} \times Y))
- (1-\tilde{\pi}_n^g(Y \times S_n\circ  g)) \notag\\
&\qquad - (1-\tilde{\pi}_n^g([p_{12}(T^{g,h}) \times X_n]_{23})) \notag\\
&= \pi_n^{-1}(S_n^{-1}) + \pi_n \circ (\id_{X_n},g)_*\mu_X(S_n\circ g)
+ \pi_n^g(p_{12}(T^{g,h})) - 2 \notag\\
&\ge 2\pi_n(S_n) + (\id_{X_n},g)_*\mu_X(g) + \pi_n^g(p_{12}(T^{g,h})) - 3 \notag\\
&\ge 1-2\varepsilon_n-\delta_n. \notag
\end{align}
By remarking the definition of $\delta_n$,
for any $g \in G_n$ there is $h \in H$
such that
\eqref{eq:subconv1} and \eqref{eq:subconv2} hold.
Also, for any $h \in H$ there is $g \in G_n$ such that
\eqref{eq:subconv1} and \eqref{eq:subconv2} hold.
This completes the proof of the theorem.
\end{proof}

As a consequence of
Theorem \ref{thm:subconv}, Proposition \ref{prop:box-boxeq},
and the completeness of $(\cX,\square)$
(see \cite{Sy:mmg}*{Theorem 4.14}),
we have the following.

\begin{cor}
$\cX_\eq$ is complete with respect to
the equivariant box metric.
\end{cor}

\section{Application}

In this section, we prove that
a sequence of lens spaces with unbounded dimension 
box converges
to the cone of an infinite-dimensional complex projective space,
as an application of Theorem \ref{thm:box-quot}.

Consider the sequence of ellipsoids:
\begin{align*}
& E_j := \{\; z \in \C^{n(j)} \mid
\sum_{i=1}^{n(j)} \frac{|z_i|^2}{\alpha_{ij}^2} = 1 \;\},\\
& \alpha_{ij} > 0, \ i=1,2,\dots,n(j), \  j = 1,2,\dots.
\end{align*}
We equip each $E_j$ with the push-forward measure of
the normalized volume measure on the unit sphere
$S^{n(j)}(1)$ by the linear transformation
and also with the restriction of the Euclidean distance function.
Then, $E_j$ is an mm-space.

Assume that $n(j) \to \infty$ as $j \to \infty$
and set $a_{ij} := \alpha_{ij}/\sqrt{n(j)}$.
We also assume $\sum_{i=1}^\infty a_i^2 < \infty$.
Then, the Gaussian space
$\Gamma^\infty_{\{a_i\}} := (H,\|\cdot\|_2,\gamma^\infty_{\{a_i\}})$ is an mm-space, where
$H := \{\,z \in \C^\infty \mid \|z\|_2 < +\infty\,\}$
is a Hilbert space and
$\gamma^\infty_{\{a_i^2\}} := \bigotimes_{i=1}^\infty (\gamma^1_{a_i^2}\otimes \gamma^1_{a_i^2})$
is the infinite product of
the one-dimensional Gaussian measure $\gamma^1_{a_i^2}$
of variance $a_i^2$.
The $U(1)$-Hopf action on $H$
(which is defined as scalar multiplication of unit complex number)
defines an mm-action on
$\Gamma^\infty_{\{a_i\}}$ and also on $E_j$.
By the homomorphism
$\Z_j := \Z/j\Z \ni [k] \mapsto e^{\sqrt{-1}k/j} \in U(1)$,
we define the mm-action of $\Z_j$ on $E_j$.
Then the quotient space $E_j/\Z_j$ is diffeomorphic to
a lens space.
The quotient space $\Gamma^\infty_{\{a_i^2\}}/U(1)$
is homeomorphic to the cone of
the infinite-dimensional complex projective space.

\begin{thm} \label{thm:lens}
If $\lim_{j\to\infty} \sum_{i=1}^{n(j)} (a_{ij}-a_i)^2 = 0$,
then, as $j\to\infty$,
\begin{enumerate}
\item $(E_j,\Z_j)$ equivariantly box converges to
$(\Gamma^\infty_{\{a_i\}},U(1))$.
\item $E_j/\Z_j$ box converges to $\Gamma^\infty_{\{a_i\}}/U(1)$.
\end{enumerate}
\end{thm}

\begin{proof}
We prove (1).
By \cite{KS:ellipsoid}*{Proposition 4.2},
the measure of $E_j$ converges to
$\gamma^\infty_{\{a_i^2\}}$ on $H$ with respect to
the $2$-Wasserstein metric.
Therefore, for any fixed $\varepsilon > 0$,
if $j$ is large enough, then
there are $\pi \in \Pi(E_j,\Gamma^\infty_{\{a_i^2\}})$ and
a closed subset $S \subset \supp\pi$ such that
$\pi(S) \ge 1-\varepsilon$ and
$\|x-x'\|_2 \le \varepsilon$ for any $(x,x') \in S$.
We observe that $d^S(g,g) \le \varepsilon$ for any $g \in U(1)$,
which together with an easy discussion proves (1).

(2) follows from (1) and Theorem \ref{thm:box-quot}.
\end{proof}

\begin{prob} \label{prob}
If $\lim_{j\to\infty}a_{ij} = a_i$,
does
$\dconc((E_j,\Z_j),(\Gamma^\infty_{\{a_i\}},U(1))) \to 0$
hold?
\end{prob}

\section{Further problem}

Gromov \cites{Gmv:green, Sy:mmg} constructs
a natural compactification of $\cX$ by using pyramids.
It is a natural problem if we have the same construction
for $\cX_\eq$.
If that is obtained, then we are able to solve Problem \ref{prob}
affirmatively.
For that, it is possible to define the Lipschitz order relation
on $\cX_\eq$ and a pyramid in a suitable way.
However, we have many obstacles.
For instance, we do not know if the limit of pyramids
is a pyramid.  Also it seems to be difficult
to formulate the concept of measurement with group action
for $\cX_\eq$.

\end{document}